\documentclass[a4paper,fleqn]{article}
\pdfoutput=1

\usepackage[intlimits]{amsmath}
\usepackage{amssymb}
\usepackage[amsmath,thref]{ntheorem}
\usepackage{calrsfs}
\usepackage[T1]{fontenc}
\usepackage{pb-diagram,lamsarrow,pb-lams}
\usepackage[inner=3cm,outer=3cm,top=3cm,bottom=3cm]{geometry}
\usepackage[T1]{fontenc}

\numberwithin{equation}{section}

\theoremstyle{change}
\theoremseparator{.\hspace{3pt}}
\newtheorem{define}{Definition}[section]
\newtheorem{thm}[define]{Theorem}
\newtheorem{prop}[define]{Proposition}
\newtheorem{cor}[define]{Corollary}
\newtheorem{lemma}[define]{Lemma}

\theoremstyle{nonumberplain}
\theoremseparator{\hspace{1pt}:}
\theorembodyfont{\upshape}
\theoremprework{}
\theorempostwork{}
\newtheorem{proof}{Proof}

\newcommand{\qed}{\vspace{-0.8\baselineskip} \hfill \rule{2.4mm}{2.4mm} }

\theoremstyle{nonumberbreak}
\theorembodyfont{\itshape}
\theoremseparator{\hspace{1pt}:}
\newtheorem{main}{Main Result}

\theoremstyle{plain}
\theoremseparator{\hspace{1pt}:}
\theoremnumbering{Alph}
\theoremheaderfont{\scshape\bfseries}
\newtheorem{theorem}{Theorem}

\newcommand{\D}{\ensuremath{\mathcal{D}}}
\newcommand{\N}{\ensuremath{\mathbb{N}}}
\newcommand{\Z}{\ensuremath{\mathbb{Z}}}
\newcommand{\R}{\ensuremath{\mathbb{R}}}

\DeclareMathOperator{\vol}{vol}
\DeclareMathOperator{\dvol}{dvol}
\DeclareMathOperator{\spec}{spec}
\DeclareMathOperator{\spa}{span}
\DeclareMathOperator{\dom}{dom}
\DeclareMathOperator{\id}{id}
\DeclareMathOperator{\im}{im}
\DeclareMathOperator{\coker}{coker}

\DeclareMathOperator{\dist}{dist}

\newcommand{\inpro}[2]{\langle #1,#2 \rangle}

\begin{document}

%---  title  -----------------------------------------------------------------------------------------------

\title{\scshape\bfseries An Adiabatic Decomposition of the Hodge Cohomology of Manifolds Fibred over Graphs\thanks{Houston Journal of Mathematics, \textbf{41} (2015), no. 1, 33--58}}
\newcommand{\shorttitle}{Decomposing the Hodge-Theory of Graph Manifolds}
\author{Karsten Fritzsch\thanks{then: Carl-von-Ossietzky Universit\"at Oldenburg; now: Leibniz Universit\"at Hannover} \\ \texttt{k.fritzsch@math.uni-hannover.de}}
\date{}

\maketitle

\vspace{2em}\begin{center}\parbox{12cm}{\begin{abstract}
In this article we use the combinatorial and geometric structure of manifolds with embedded cylinders in order to develop an adiabatic decomposition of the Hodge cohomology of these manifolds. We will on the one hand describe the adiabatic behaviour of spaces of harmonic forms by means of a certain \emph{\v{C}ech-de Rham} complex and on the other hand generalise the \emph{Cappell-Lee-Miller splicing map} to the case of a finite number of edges, thus combining the topological and the analytic viewpoint. In parts, this work is a generalisation of works of Cappell, Lee and Miller in which a single-edged graph is considered, but it is more specific since only the Gauss-Bonnet operator is studied.
\end{abstract}}\end{center}

\thispagestyle{empty}

\vspace{2em}

%---  introduction  ----------------------------------------------------------------------------------------
\section{Introduction}

Let $(X,g)$ be a smooth, oriented and closed Riemannian manifold and suppose there is a finite number of codimension $1$ submanifolds $Y_e$ so that $g$ is of product type near each $Y_e$. Then, the location of the submanifolds $Y_e$ can be used to define a finite graph $\mathcal{G}=(V,E)$. Let $F=\Lambda(T^*X)$ be the exterior algebra of $X$ and $D = d + \delta$ be the Gauss-Bonnet operator on the $\R$-module $C^\infty(F)$ of smooth differential forms. The local product structure of $g$ near the submanifolds $Y_e$ admits inserting cylinders $Y_e \times [-r,r]$ at $Y_e$, thus stretching the manifold $X$ to a prolonged version $X(r)$, in such a way that we may canonically extend the Riemannian metric, as well as the bundles and the operators.

In the course of this article we will partially answer the following questions:
\begin{quote}\emph{In which way does the space $\mathcal{H}(r)=\ker D(r)$ of harmonic forms on $X(r)$ behave for increasing $r$ and is there -- in some sense yet to be defined -- a limit as $r \rightarrow \infty$? How does the location and the geometry of the cylinders enter the picture?}\end{quote}

The main tool of study on the analytic side is the \emph{Cappell-Lee-Miller splicing map} (cf. \cite{clm1}), gluing together particular harmonic forms on the disconnected manifold $X(\infty)$ and yielding forms on $X(r)$; on the topological side there is the \emph{Mayer-Vietoris principle}, applied to a cover subordinate to the graph $\mathcal{G}$. We will actually see that there is a close correspondence between these seemingly different approaches.

This is a first -- and vague -- version of the main result which, for a single-edged graph, was previously shown by Cappell, Lee and Miller \cite{clm1}:
\begin{main}[vague]\label{main.vague}
 \begin{enumerate}
  \item The Gauss-Bonnet operator has no non-zero small eigenvalues, i.e. eigenvalues $\mu \in o(\frac{1}{r})$.
  \item The splicing construction yields forms which are exponentially close to harmonic forms. Every harmonic form can be approximated in this way.
  \item The topological representation of the cohomology of $X(r)$ by means of a \v{C}ech-de Rham complex for the graph $\mathcal{G}$ gives the asymptotics of the Hodge cohomology of $X(r)$ as $r \rightarrow \infty$.
 \end{enumerate}
\end{main}

This work is influenced by two ancestral ideas from the analysis on singular spaces: Stretching a single embedded cylinder and linking the adiabatic behaviour of the Hodge cohomology to exact sequences in de Rham cohomology (as did Cappell, Lee and Miller \cite{clm1}) or spectral sequences (compare the article of Mazzeo and Melrose \cite{mm90}), and considering multiple cylinders but restricting the theory to functions (as in Grieser \cite{grieser08}). The first of the above questions has been studied for a fibration of a compact manifold by Mazzeo and Melrose \cite{mm90}, who developed a very detailed analytic theory. (In the case of a one edge graph and an asymptotically cylindrical structure, Hassell, Mazzeo and Melrose \cite{hmm95} give detailed asymptotics for the eigenvalues -- and other objects -- as $r$ tends to $\infty$.) For the Laplacian on functions, the second question was treated extensively by Grieser \cite{grieser08}. Moreover, one may try to construct harmonic forms by means of a splicing construction, which involves combining $L^2$ harmonic forms on $X(\infty)$ with harmonic forms defined on the cross-sections of the cylinders. This approach has been pursued by Cappell, Lee and Miller \cite{clm1}, for instance.

We basically follow the lines of \cite{clm1}, but treat the topological aspect -- i.e. the Mayer-Vietoris sequences -- more extensively: Considering the singular space $\mathcal{G}$ as the base involves using its simplicial cohomology in order to account for its singular structure. Moreover, the final distance argument leading to surjectivity of projected splicing will be carried out explicitly, in contrast to \cite{clm1}. (Nicolaescu \cite{nicolaescu02} uses this argument as well, but in a more abstract situation/language.)

\bigskip The first section of this article will see the exact definition of manifolds fibred over graphs as well as decompositions of differential forms and of the Gauss-Bonnet operator with respect to cylindrical metrics. Following this, suitable boundary conditions will be given and after introducing extended $L^2$ forms and their limiting values in Section \ref{sec:extendeds}, some well-known results on the Gauss-Bonnet operator, its kernel and the spaces of limiting values will be collected.

Then, in Section \ref{sec:gen.mv} the conditions which a set of harmonic forms on the cover $X^0$ needs to satisfy in order to be considered for splicing are given; the specific definition of the matching conditions is motivated by a morphism occurring in a generalised Mayer-Vietoris sequence. In Section \ref{sec:match.ext} this connection between harmonic forms and matching sets on one side and the generalised Mayer-Vietoris sequence on the other side will be studied. A theorem from Atiyah, Patodi and Singer \cite{aps1} along with related results will serve as a starting point.

Finally, in the fifth section, we introduce the Cappell-Lee-Miller splicing map and at this point it is reasonable to reformulate the main result in an exact fashion (as {\scshape\bfseries Theorem A} in Section \ref{sec:main}). Some preliminary work concerning the splicing map -- convergence results (see Section \ref{sec:convergence}) established earlier by Cappell, Lee and Miller \cite{clm1} -- will follow. This enables us to use a distance argument to show that splicing defines an isomorphism in Section \ref{sec:surjective}.

%---  acknowledgements  ------------------------------------------------------------------------------------
\bigskip\textbf{Acknowledgements.} The author would like to thank Daniel Grieser for suggesting this problem, he is deeply grateful for the many fruitful discussions and comments.

%---  D and fibred manifolds  ------------------------------------------------------------------------------
\section{The Gauss-Bonnet Operator on Fibred Manifolds}\label{sec:setting}

To begin with, we introduce \textbf{manifolds fibred over a graph}, in order to provide the setting in which the main part takes place. This or similar settings have been used frequently to decompose different objects into contributions from different parts of the manifold. (For instance cohomologies, eigenspaces or $\zeta$-determinants: examples would be \cite{aps1},\cite{grieser08}, \cite{pw06} and many more.) The setup goes along the lines of \cite{grieser08}.

%---  fibred manifolds  ------------------------------------------------------------------------------------
\bigskip\textbf{Fibred Manifolds.} Let $(X,g)$ be a smooth, closed and oriented Riemannian manifold of dimension $n>1$ and $(Y_e,g_e)$, $e \in E$, be a finite set of disjoint and closed submanifolds of codimension $1$ endowed with the induced metric. Moreover, we suppose that $X$ is of \textbf{product type} near each $Y_e$, i.e. there exist isometries
  \[  \varphi_e : Y_e\times(-1,1) \longrightarrow X,  \]
where the domain of $\varphi_e$ is equipped with the product metric $g_e+dt^2$. Let the $Y_e$ carry the orientation induced by those of $X$, $(-1,1)$ via $\varphi_e$.

Let $X_v$, $v \in V$, denote the connected components of $X\setminus\left(\bigcup Y_e\right)$ and $\overline{X_v}$ their closures. A pair $(v,e)$ is called a \textbf{half-edge}, if $Y_e\subset\overline{X_v}$ and in this case $v$ is said to be adjacent to $e$: $v \sim e$. That is, an edge $e$ connects two vertices $v,v^\prime$ if and only if
   \begin{equation*}
      Y_e\subset\left(\overline{X_v}\cap\overline{X_{v^\prime}}\right).
   \end{equation*}
We will often call the connected components $X_v$ \textbf{vertex manifolds} and the submanifolds $Y_e$ \textbf{cross-sections}. In this way, a finite graph $\mathcal{G}=(V,E)$ is defined and the isometries $\varphi_e$ determine an orientation of $\mathcal{G}$. Let $e=(v,v^\prime)$ denote the oriented edge running from $v$ to $v^\prime$ and define
  \begin{equation}\label{bullet}
     \bullet_{(v,e)} = \left\{\begin{array}{cl} 1 &\text{ for } e=(v,v^\prime) \\
                                               -1 &\text{ for } e=(v^\prime,v) \,,
                              \end{array}\right.
  \end{equation}
for any half-edge $(v,e)$. Since $X_v$ is of product type near each adjacent $Y_e$, for $r \ge 0$ it is possible to use the isometries $\varphi_e$ to glue cylinders $Z_{(v,e)}(r) = Y_e \times [0,-\bullet_{(v,e)}r]$ to the corresponding subset of the boundary of $\overline{X_v}$, identifying $Y_e \subset \overline{X_v}$ with $Y_e \times \{-\bullet_{(v,e)}r\}$. (An interval $[a,b]$ with $a>b$ is understood to be the interval $[b,a]$.) Thus, we obtain \textbf{prolonged vertex manifolds} $X_v(r)$ and identifying parts of the boundaries of the $X_v(r)$ according to the graph structure yields the \textbf{prolonged manifold} $X(r)$. Note, that the embedded cylinders in $X(r)$ correspond to $Z_e(r) = Y_e \times [-r,r]$.

If instead of finite cylinders we use $Y_e \times \R_{\bullet_{(v,e)}}$ in the first step, this time identifying $Y_e \subset \overline{X_v}$ with $Y_e \times \{0\}$, the same process yields the \textbf{stars} $X_v(\infty)$ of the vertex manifolds. We use the following notations:
  \begin{equation*}\begin{array}{cclcccl}
    Y_v         &:=  &\displaystyle\bigcup_{e\sim v}  Y_e  &,
   &Y           &:=  &\displaystyle\bigcup_{e\in E}   Y_e \\[1.2em]
    Z_v(r)      &:=  &\displaystyle\bigcup_{e\sim v}  Z_{(v,e)}(r)  &,
   &Z(r)        &:=  &\displaystyle\bigcup_{e\in E}   Z_e(r) \\[1.2em]
    Z_v(\infty) &:=  &\displaystyle\bigcup_{e\sim v}  Z_{(v,e)}(\infty) &,
   &Z(\infty)   &:=  &\displaystyle\bigcup_{e\in E}   Z_e(\infty) \;,
  \end{array}\end{equation*}
Later on, we will sometimes write $Z^0(s)$ instead of $Z(s)$ for $s\in[0,\infty]$. This is closer to the notation of section \ref{sec:match.ext}.

Furthermore, the product type structure makes it possible to canonically extend the Riemannian metric $g$ to the prolonged (vertex) manifolds and cylinders. This metric will be assumed as given for the rest of this article.

%---  gauss-bonnet  ----------------------------------------------------------------------------------------
\bigskip\textbf{The Gauss-Bonnet Operator.} $F$ will denote the exterior algebra $\Lambda T^*X$ of $X$; similarly, we denote the corresponding bundles over $X(r)$, $X_v(r)$ etc. by $F(r)$, $F_v(r)$ etc., whereas $F_e = \Lambda T^*Y_e$. $C^\infty(F)$ will denote the space of smooth sections of $F$ -- that is, smooth (differential) forms on $X$. Moreover, the Riemannian metrics on the manifolds induce Riemannian metrics on the bundles. The corresponding inner products will always be denoted by $\inpro{\cdot}{\cdot}$, with suitable subscripts if the domain is not clear from context. The inner products allow for defining spaces of square-integrable sections, e.g. $L^2(F)$, and Sobolev spaces of sections, e.g. $H^k(F)$. Defining
  \[ \widehat{F}_e = \Lambda T^*Y_e \oplus \Lambda T^*Y_e \;, \]
and $\widehat{F}_v = \bigoplus_{e \sim v} \widehat{F}_e$, $\widehat{F} = \bigoplus_{e \in E} \widehat{F}_e$, it is plain to see that there is an isometric isomorphism of vector bundles
  \begin{equation*}
     F\big|_{Z_e(r)} \cong C^\infty\big([-r,r],\widehat{F}_e)
     \quad,\quad u=u^0 + u^1\wedge dt \longmapsto \begin{pmatrix} u^0 \\ u^1 \end{pmatrix} \;,
  \end{equation*}
where $t$ denotes the longitudinal coordinate of the cylinder and $u^j : [-r,r] \rightarrow C^\infty(F_e)$ is smooth.

Let $d : C^\infty\big(F(r)\big) \rightarrow C^\infty\big(F(r)\big)$ be the exterior derivative and $\delta = d^*$ be its formal adjoint (with respect to the $L^2$ inner product). Since $X(r)$ is closed, the \textbf{Gauss-Bonnet operator}
  \[ D(r) = d + \delta \;:\; C^\infty\big(F(r)\big) \longrightarrow C^\infty\big(F(r)\big) \]
is a linear, essentially self-adjoint, elliptic first-order partial differential operator which extends to a closed, self-adjoint operator on the first Sobolev space
  \[ \D(r) \;:\; H^1\big(F(r)\big) \longrightarrow L^2\big(F(r)\big) \;. \]
Its symbol is given in terms of left exterior multiplication $\wedge_\xi u = \xi \wedge u$ and its adjoint, the interior product $\iota_\xi u = u(v,\dotsc)$, where $v$ is the vector field dual to $\xi$ with respect to the chosen coordinates: $p_{\D(r)}(x,\xi) = i ( \wedge_\xi - \iota_\xi )$, where $i$ denotes the imaginary unit.

In addition, $L^2\big(F(r)\big)$ possesses an orthonormal basis of eigenforms of $\D(r)$, which are smooth and the spectrum $\spec\D(r)$ of $\D(r)$ is a discrete subset of $\R$ with sole accumulation points $\pm \infty$ and which consists of eigenvalues of finite multiplicity only.

\bigskip On $X_v(\infty)$, we define the exterior differential $d$ with domain $\dom d= C^\infty_0\big(X_v(\infty)\big)$, and with respect to the $L^2$ inner product the codifferential $\delta$ and hence the Gauss-Bonnet operator $D_v(\infty)$ can be defined. Considered with domain $\dom D_v(\infty) = C_0^\infty\big(F_v(\infty)\big)$, Gaffney \cite{gaffney51} showed that $D_v(\infty)$ is essentially self-adjoint and that its closure $\D_v(\infty)$ has domain $H^1\big(F_v(\infty)\big)$.

\bigskip The cylindrical structure of the manifold $X$ near the cross-sections $Y_e$ of course has its consequences for the operator $D$ as well. In the following, the existence of an isometric isomorphism from a suitable neighbourhood of $Y_e$ to a cylinder $Y_e \times I$, $I \subset \R$ is used implicitly.

\begin{prop}\label{D.on.cylinders}Let $D$ be the Gauss-Bonnet operator over a cylinder $Y_e \times I$ and $\pi$ be projection onto the first factor of $Y_e \times I$. There is an isometric bundle isomorphism on $\widehat{F}_e$ inducing
  \[ \sigma \;:\; C^\infty(\widehat{F}_e) \longrightarrow C^\infty(\widehat{F}_e) \]
and a linear, elliptic, symmetric partial differential operator
  \[ \widehat{D}_e \;:\; C^\infty(\widehat{F}_e) \longrightarrow C^\infty(\widehat{F}_e) \]
so that
  \begin{equation}\label{D.decomp}
    D = (\pi^*\sigma)\big(\partial_t - \pi^*\widehat{D}_e\big) \;,
  \end{equation}
where $t$ denotes the coordinate of $I$.
\end{prop}

\begin{proof} Recall that there is a unique isometric isomorphism $\ast \;:\; C^\infty(F) \longrightarrow C^\infty(F)$ so that $\inpro{\ast u}{v} = \inpro{u\wedge v}{\dvol}$. This isometry is called the \textbf{Hodge-$\ast$ operator} and when acting on $k$-forms defined on a closed, $n$-dimensional manifold, $\ast$ satisfies the identities
  \[ \delta = (-1)^{n(k+1)+1} * d * \quad\text{and}\quad *^2 = (-1)^{k(n-k)} \;. \]
Using these relations and local coordinate expressions we may examine the relation between the operators $d$, $\delta$ and $\ast$ and their counterparts defined on the cross-sections $Y_e$:
  \begin{equation*}\begin{array}{rclcrcl}
    du     &= &d_e u + (-1)^k(\partial_t u)\wedge dt &, &d (u \wedge dt)      &= &d_e u \wedge dt \\[3pt]
    \ast u   &= &(\ast_e u)\wedge dt                   &, &\ast(u \wedge dt)    &= &(-1)^{n-k-1}\ast_e u \\[3pt]
    \delta u &= &\delta_e u                            &, &\delta( u\wedge dt ) &= &(-1)^{k+1}\partial_t u
                                                                                   + \delta_e u \wedge dt ,
  \end{array}\end{equation*}
where $u \in C^\infty(I) \otimes C^\infty(F_e)$ is of degree $k$, a subscript $e$ for operators denotes the corresponding operator on $C^\infty(F_e)$ and $\dim Y = n-1$. Then, with $(-1)^{\deg_Y}u =(-1)^k u$ for a homogeneous $k$-form $u \in L^2(F_e)$ and extending this definition by linearity to all of $L^2(F_e)$, we define
  \begin{equation}\label{D.hat.sigma.def}\begin{split}
    \sigma &:= \begin{pmatrix}0 & (-1)^{\deg_Y+1} \\ (-1)^{\deg_Y} & 0 \end{pmatrix} \\[0.3em]
    \widehat{D}_e &:=\sigma \begin{pmatrix} D_e & 0 \\ 0 & D_e \end{pmatrix}\;.
  \end{split}\end{equation}
Now, it is easy to see that \eqref{D.decomp} holds. With regard to the remaining claims, it is convenient to collect some properties of $\sigma$ and $\widehat{D}_e$.
\qed\end{proof}

\begin{lemma}\label{props.of.sigma.D.hat} Let $\sigma$ be defined as in \eqref{D.hat.sigma.def} and $(-1)^{\deg} u = (-1)^{\deg_Y}u^0 + \big((-1)^{\deg_Y + 1}u^1\big) \wedge dt$. Then,
 \begin{enumerate}
  \item \quad $(-1)^{\deg+1} \begin{pmatrix}0&1\\1&0\end{pmatrix} = \sigma
           = \begin{pmatrix}0&1\\1&0\end{pmatrix} (-1)^{\deg} $
  \item \quad $\sigma ^2 = -\id \quad \text{ and } \quad \sigma^* = -\sigma$
  \item \quad $\begin{pmatrix}D_e&0\\0&D_e\end{pmatrix} \sigma
          = - \sigma \begin{pmatrix}D_e&0\\0&D_e\end{pmatrix} \;,$
 \end{enumerate}
$\widehat{D}_e$ as in \eqref{D.hat.sigma.def} is a linear, symmetric, elliptic first-order differential operator and when considered as an unbounded operator on $L^2(\widehat{F}_e)$, it extends to a closed operator $\widehat{\D}_e$ with domain $H^1(\widehat{F}_e)$. Moreover,
 \begin{enumerate}\setcounter{enumi}{3}
  \item $L^2(\widehat{F}_e)$ has an orthonormal basis $\{\phi_\lambda\}$ of smooth eigenforms of
        $\widehat{\D}_e$ with eigenvalues $\lambda\in\R$, $\widehat{D}_e\phi_\lambda=\lambda\phi_\lambda$;
        the spectrum of $\widehat{D}_e$ is a discrete subset of $\R$ which consists of eigenvalues of finite multiplicity only and whose sole accumulation points are $\pm\infty$.
  \item The eigenvalues of $\widehat{D}_e$ can be indexed by $k \in \Z\setminus\{0\}$ such that $\dotsc \le \lambda_{k-1} \le \lambda_k \le \dotsc$, counted with multiplicities and $-\lambda_k = \lambda_{-k}$. In particular, $\ker\widehat{D}_e$ is even-dimensional.
  \item As $k \rightarrow \infty$, the eigenvalues of $\widehat{D}_e$ satisfy
           \[ |\lambda_{\pm k}| \sim
                  \left(\frac{\vol(Y_e)\vol(\mathbb{S}^{n-2})}{n-1}\right)^{-\frac{1}{n-1}} k^{\frac{1}{n-1}} \;, \]
        with $\mathbb{S}^{n-2}$ denoting the unit sphere in $\R^{n-1}$. (This means that the quotient of
        both sides tends to $1$ as $k \rightarrow \infty$.)
  \end{enumerate}
\end{lemma}

\begin{proof} Items \emph{i)} through \emph{iii)} follow from direct computations. For item \emph{iv)}, the reader is referred to Shubin \cite[8.2]{shubin01}. Concerning item \emph{v)}, note that $\sigma\widehat{D}_e = -\widehat{D}_e\sigma$ implies that $\sigma\phi_\lambda$ is an eigenform with eigenvalue $-\lambda$, hence $\sigma$ induces an isomorphism between the eigenspaces corresponding to $\lambda$ and $-\lambda$. Moreover, $\ker\widehat{D}_e \cong \big(\ker D_e\big)^2$ shows the even-dimensionality of $\ker\widehat{D}_e$. For item \emph{vi)}, compare Gilkey \cite[3.7]{gilkey04}. Using Gilkey's trace formula for the heat operator $e^{-t\widehat{\Delta}_e}$ of the direct sum of two Hodge-Laplacians $\widehat{\Delta}_e = \widehat{D}_e^2 = \Delta_e \oplus \Delta_e$, one may prove that
 \[ \lim_{\lambda \rightarrow \infty} \lambda^{-\frac{n-1}{2}}N_{\widehat{\Delta}_e}(\lambda)
     = \frac{2\pi^\frac{n-1}{2}\vol(Y_e)}{\Gamma(\frac{n+1}{2})} \;, \]
where $\Gamma$ denotes the Gamma-function and $N_{\widehat{\Delta}_e}(\lambda)$ the number of eigenvalues of $\widehat{\Delta}_e$ less or equal $\lambda$ (counted with multiplicities). This shows
  \[ N_{\widehat{D}_e}^+(\lambda)
     = \frac{\vol(Y_e)\vol(\mathbb{S}^{n-2})}{n-1}\lambda^{n-1} + \hbox{o}(\lambda^{n-1}) \;, \]
with $N_{\widehat{D}_e}^+$ the number of non-negative eigenvalues of $\widehat{D}_e$ less or equal $\lambda$. Since $\widehat{\Delta}_e = D_e^2$, item \emph{vi)} follows directly from these asymptotics.
\qed\end{proof}

%---  fredholm properties  ---------------------------------------------------------------------------------
\bigskip\textbf{Fredholm Properties.} The kernels of the operators $D_v(r)$ and $D_v(\infty)$ will play a crucial role, hence we need to study the properties of these operators and compare their kernels. Fix an orthonormal basis of $L^2(\widehat{F}_e)$ of eigenforms of $\widehat{D}_e$ as in \thref{props.of.sigma.D.hat}, say $\{\phi^e_k\}$ and suppose that for $0 < |k| \le N_e$, the forms $\phi^e_k$ span $\ker \widehat{D}_e$. Then, the union of these bases forms a basis of $L^2(\widehat{F})$, where
  \[ \spa \big\{\,\phi_{k_e}^e\;\big|\; 0 < |k_e| \le N_e\,,\, e \in E\,\big\} = \ker\widehat{D} \;. \]
Simplifying notation, we denote the joint basis by $\phi_k$, $k \in \Z\setminus 0$, put $\sum_e N_e = N$ and order this basis in such a way that $\lambda_k \le \lambda_l$ if and only if $k \le l$ and
  \[ \spa \big\{\,\phi_k\;\big|\; 0 < |k| \le N\,\big\} = \ker\widehat{D} \;. \]
For the sake of clarity, we will often denote an element of this basis corresponding to the eigenvalue $\lambda$ by $\phi_\lambda$, even though there may be multiple elements corresponding to this eigenvalue.

Now there are two different representations of a form $u$ on the cylindrical parts of the manifold: We may decompose $u$ into its \textbf{transversal} (or \textbf{absolute}) and its \textbf{normal} (or \textbf{relative}) part as
  \begin{equation}\label{normal.form.decomp}
    u = u^0 + u^1\wedge dt
  \end{equation}
or we may expand $u$ in terms of the basis $\{\phi_\lambda\}$ as
  \begin{equation}\label{forms.expand}
    u = \sum_\lambda u_\lambda \pi^*\phi_\lambda \;.
  \end{equation}
Pulling back along the inclusion map, e.g. $\partial Z_{(v,e)}(r) \hookrightarrow X_v(r)$, does not comprise the relative part of a form; hence, by restriction to the boundary we understand evaluation at $t=0$ (note that the coordinate $t$ defines a boundary defining function):
  \[ u\big|_{\partial X_v(r)} = u^0(0) + u^1(0) \wedge dt = \sum_\lambda u_\lambda(0) \phi_\lambda \;. \]
If $P$ is a subspace of $\widehat{F}_v$, let $C^\infty\big(F_v(r);P\big)$ denote the space of forms $u \in C^\infty\big(F_v(r)\big)$ so that
   \[ u\big|_{\partial X_v(r)} \in P \;. \]
Then, for an operator $T$ acting on $C^\infty\big(F_v(r)\big)$, $\ker(T;P)$ will denote the kernel of $T$ with domain $C^\infty\big(F_v(r);P\big)$.

Since we are mainly interested in the behaviour of the kernel of $D_v(r)$ and $D(r)$ as $r \rightarrow \infty$, it is reasonable to assume that the forms satisfy certain growth conditions on the cylindrical parts. Restricting forms to the cylindrical part of $X_v(r)$ for the moment, solving $D_v(r) u = 0$ is equivalent to solving $\partial_t u_\lambda(t) = \lambda u_\lambda(t)$ for all eigenvalues $\lambda$, by \eqref{D.decomp} and \eqref{forms.expand}. Hence, any solution is of the form
  \begin{equation*}
    u(y,t)\big|_{Z_{(v,e)}(r)} = \sum_{\lambda \,\in\, \spec(\widehat{D}_e)} u_{\lambda}(0) \, e^{\lambda t}
                                                                           \, \pi^*\phi_\lambda(y)
  \end{equation*}
This motivates the following definition of boundary conditions:

\begin{define} Let $P_{e,\pm}$ denote the $L^2$ closure of the space spanned by eigenforms of $\widehat{D}_e$ corresponding to positive respectively negative eigenvalues and let $\overline{P}_{e,\pm} = P_{e,\pm} \oplus \ker \widehat{D}_e$. We define spaces of boundary conditions by
  \begin{equation*}\begin{array}{cclcccll}
    P_{(v,e)} &:= &P_{e,\bullet_{(v,e)}} &\text{and}
    &\overline{P}_{(v,e)} &:= &\overline{P}_{e,\bullet_{(v,e)}} &,\\[3pt]
    P_v &:= &\displaystyle\bigoplus_{e\sim v} P_{(v,e)} &\text{and}
    &\overline{P}_v &:= &\displaystyle\bigoplus_{e\sim v}\overline{P}_{(v,e)} &.
  \end{array}\end{equation*}
\end{define}
Note, that the resulting boundary conditions give the \textbf{Atiyah-Patodi-Singer conditions} with respect to $-\widehat{D}$. As is well-known, these boundary conditions define Fredholm operators (cf. \cite[\S 2]{aps1}):

\begin{thm}[Atiyah-Patodi-Singer] The operator
  \begin{equation*}
    D_v(r) : C^\infty\big(F_v(r);P_v\big) \longrightarrow C^\infty\big(F_v(r)\big)
  \end{equation*}
extends to a closed and continuous operator
  \begin{equation*}
    \D_v(r) : H^1\big(F_v(r);P_v\big) \longrightarrow L^2\big(F_v(r)\big) \;.
\end{equation*}
In addition, these operators are Fredholm and their kernels coincide:
  \[\ker\big(\D_v(r);P_v\big)=\ker\big(D_v(r);P_v\big) \]
\end{thm}

%---  extended forms  --------------------------------------------------------------------------------------
\section{Extended $L^2$ Forms}\label{sec:extendeds}

On the stars $X_v(\infty)$, by restricting our attention to $L^2$ forms we lose information since the \textbf{constant part} $\sum_{\lambda=0} u_\lambda \pi^*\phi_\lambda$ has to vanish identically for the form to be square-integrable. This is why we need to introduce the well-known concept of \textbf{extended $\mathbf L^2$ forms}.

\begin{define} An \textbf{extended $L^2$ form} on $X_v(\infty)$ (or more generally, on a manifold with cylindrical ends) is a pair $(u,\widehat{u})$ where $u$ is a form on $X_v(\infty)$, $\widehat{u} \in L^2(\widehat{F}_v)$ and
  \begin{enumerate}
   \item $u \in L^2_{loc}\big(F_v(\infty)\big)$\; , \;$\widehat{u} \in \ker\widehat{D}_v$
   \item $u \big|_{Z_v(\infty)} - \pi^*\widehat{u} \in L^2\big( F_v(\infty)\big|_{Z_v(\infty)} \big)$
  \end{enumerate}
Here, the operator $\widehat{D}_v$ denotes the direct sum $\bigoplus_{e \sim v} \widehat{D}_e$. We denote the space of extended $L^2$ forms by $L^2_{ex}\big(F_v(\infty)\big)$.

The form $\widehat{u}$ can be regarded as an asymptote for $u$, hence will be dubbed the \textbf{limiting value} of $u$.
\end{define}

Alternatively (as is done in \cite{aps1}), an extended $L^2$ form can be considered a single form $\widetilde{u}$, which is locally square-integrable, so that for $|t|$ large we have
  \begin{equation}\label{eq:ext.representatives}
   \widetilde{u}(y,t) = u(y,t) + \pi^*\widehat{u}(y) \;,
  \end{equation}
where $u \in L^2\big(F_v(\infty)\big|_{Z_v(\infty)}\big)$ and $\widehat{u} \in \ker \widehat{D}_v$. This motivates the following definition.

\begin{define}\label{def:ext.gb} Let $\D_v^{ex}(\infty)$ be the Gauss-Bonnet operator with domain comprising of the extended $L^2$ forms with representatives $u \in \dom\D_v(\infty)$, $\widehat{u} \in \ker\widehat{D}_v$ (with respect to \eqref{eq:ext.representatives}), i.e.
 \[ \dom \D_v^{ex}(\infty) \cong \dom \D_v(\infty) \oplus \ker\widehat{D}_v \;. \]
Forms satisfying $\D_v^{ex}(\infty) u = 0$ are called \textbf{extended $L^2$ solutions}.
\end{define}

After introducing extended $L^2$ forms, explicit calculations easily verify a comparison of kernels which was published in \cite[3.11]{aps1} as well as \cite[2.2]{clm1}:

\begin{prop}[Atiyah-Patodi-Singer]\label{exts.and.kers}
  \begin{equation}\label{isos.kernels}\begin{array}{lccl}
     &\ker\big(D_v(r);P_v\big)            &\cong &\ker \D_v(\infty) \\[3pt]
     &\ker\big(D_v(r);\overline{P}_v\big) &\cong &\ker\D_v^{ex}(\infty)
  \end{array}\end{equation}
Further, for an extended $L^2$ solution $(u,\widehat{u})$ of $\D_v^{ex}(\infty)u=0$ there is $\widetilde{u} \in \ker\big(D_v(r);\overline{P}_v\big)$ such that
  \begin{equation*}\begin{array}{rcl}
    u(y,t)\big|_{Z_v(\infty)} &= &\bigoplus\limits_{e \sim v}\sum\limits_{\bullet \lambda_e \ge 0}
                                  \widetilde{u}_{\lambda_e}(0)e^{\lambda_e t}\pi^*\phi_{\lambda_e}(y)
             \qquad\text{and} \\[1em]
    \widehat{u} &= &\bigoplus\limits_{e \sim v} \sum\limits_{\lambda_e=0} \widetilde{u}_{\lambda_e}(0) \,
                    \phi_{\lambda_e} \qquad\text{in}\quad C^\infty(\widehat{F}_v)\;,
  \end{array}\end{equation*}
where the $\widetilde{u}_{\lambda_e}$ denote the coefficient functions in the expansion of $\widetilde{u}$ on $Z_{(v,e)}(r)$ with respect to the orthonormal basis $\{\phi_{\lambda_e}\}$ of eigenfunctions of $\widehat{D}_e$.
\end{prop}
The existence of isomorphisms \eqref{isos.kernels} shows that there are natural restriction maps
  \begin{equation*}\begin{array}{rccc}
   B_v : &\ker\D_v^{ex}(\infty) \longrightarrow \ker\widehat{D}_v &\;, &(u,\widehat{u}) \longmapsto \widehat{u} \;,
  \end{array}\end{equation*}
which we call \textbf{boundary data maps}. With respect to the representation $\pi^*\widehat{u} = u^0 + u^1 \wedge dt$, we define \textbf{absolute} and \textbf{relative} boundary data maps $B^a_v$, $B^r_v$ by
  \begin{equation*}\begin{array}{rccc}
   B^a_v : &\ker\D_v^{ex}(\infty) \longrightarrow \ker D_{Y_v} &\;, &(u,\widehat{u}) \longmapsto u^0 \;, \\[3pt]
   B^r_v : &\ker\D_v^{ex}(\infty) \longrightarrow \ker D_{Y_v} &\;, &(u,\widehat{u}) \longmapsto u^1 \;.
  \end{array}\end{equation*}

\begin{define} We denote the images of the boundary data maps $B_v$ (respectively $B_v^a$, $B_v^r$) by $L_v$ ($L_v^a$, $L_v^r$) and call them (\textbf{absolute} respectively \textbf{relative}) \textbf{limiting values}.
\end{define}
If for $u \in \ker\D_v^{ex}(\infty)$, $w \in \ker\D_v(\infty)$ we define
  \begin{equation}\label{ext.inpro}
    \inpro{u}{w}_{ex} \; = \int_{X_v(\infty)} u\wedge\ast w \;,
  \end{equation}
this gives a pairing between $\ker\D_v^{ex}(\infty)$ and $\ker\D_v(\infty)$. Let $w_1, \dotsc, w_k$ be an orthonormal basis of $\ker\D_v(\infty)$ (this space is finite dimensional by \thref{exts.and.kers}). It is easy to verify that the definitions of the \textbf{extended part} $[u]_{ex} := u - \sum \inpro{u}{w_j}_{ex} w_j$ and the $\mathbf L^2$ \textbf{part} $[u]_{L^2} := \sum \inpro{u}{w_j}_{ex}w_j$ do not depend on the choice of basis and that the decomposition $u = [u]_{ex} + [u]_{L^2}$ is orthogonal with respect to the pairing $\inpro{\cdot}{\cdot}_{ex}$ in the sense that for any $u \in \ker\D_v^{ex}(\infty)$, $w \in \ker\D_v(\infty)$ we have $\inpro{[u]_{ex}}{w}_{ex} = 0$.

Having this decomposition and pairing \eqref{ext.inpro} in mind, we can show that the extended part of an extended $L^2$ solution is uniquely determined by its limiting value. This implies that the space of extended $L^2$ solutions splits as a direct sum
  \begin{equation*}\begin{split}
    \ker\D_v^{ex}(\infty) &= \big\{\, [u]_{L^2} \;\big|\; u \in \ker\D_v^{ex}(\infty) \,\big\} \oplus
                            \big\{\, [u]_{ex}  \;\big|\; u \in \ker\D_v^{ex}(\infty) \,\big\} \\[2pt]
                          &\cong \; \ker\D_v(\infty) \oplus L_v \;.
  \end{split}\end{equation*}
And as a simple corollary of this and \thref{exts.and.kers} we obtain:

\begin{cor}[Atiyah-Patodi-Singer] The $L^2$ solutions of \,$\D_v(\infty)u=0$ and $\D_v^2(\infty)u=0$ coincide, as do the extended $L^2$ solutions of $\D_v^{ex}(\infty)u=0$ and $\big(\D_v^{ex}(\infty)\big)^2u=0$. Furthermore, the spaces $\ker\big(\D_v(r);\overline{P}_v\big)$, $\ker\D_v(\infty)$ and $\ker\D_v^{ex}(\infty)$ are finite dimensional.
\end{cor}

Since $D(r)$ is essentially self-adjoint and elliptic, it is an easy exercise to show that $\ker D(r) = \ker \D^2(r) = \ker \overline{\Delta}(r) = \ker \Delta(r)$, where $\Delta$ denotes the \textbf{Hodge-Laplace operator} on differential forms. $\D_v(\infty)$ is not self-adjoint, but taking into account the second part of \thref{exts.and.kers}, nevertheless we may show $\ker\D_v(\infty) = \ker \Delta\big|_{L^2(F_v(\infty))}$ and $\ker\D_v^{ex}(\infty) = \ker\Delta\big|_{L^2(F_v(\infty))}$. (Here, $\Delta\big|_{L^2(F_v(\infty))}$ is considered the Hodge-Laplace operator on extended $L^2$ forms, compare \thref{def:ext.gb}.) Since the term Hodge cohomology refers to the space of \textbf{harmonic forms} (i.e. forms being in the kernel of the Hodge-Laplacian), this is why we sometimes refer to \textbf{extended harmonic forms} instead of extended $L^2$ solutions.

\bigskip Finally, we want to comment on a symplectic structure of $L_v$. Let
  \begin{equation}\label{symplectic.product}
    \{\cdot,\cdot\}_v \;:\; L^2(\widehat{F}_v) \times L^2(\widehat{F}_v) \longrightarrow \R
     \quad,\quad \{u,w\}_v := \sum_{e \sim v} \bullet_{(v,e)} \inpro{u}{\sigma w}_{\widehat{F}_e} \;.
  \end{equation}
Then, $\{\cdot\,,\cdot\}_v$ defines a symplectic inner product on $L^2(\widehat{F}_v)$, which restricts to a symplectic inner product on $\ker\widehat{D}_v$. In the literature (for instance in \cite[2.3]{clm1} which uses \cite{aps1}) there are various proofs of the fact that with respect to inner product \eqref{symplectic.product} $L_v$ is a Lagrangian subspace of $\ker\widehat{D}_v$. But what is of more importance is the resulting decomposition of the spaces of limiting values.

\begin{prop}[Cappell-Lee-Miller]\label{limiting.values.decomp}The spaces of limiting values split as direct sums $L_v = L_v^a \oplus L_v^r$ and it holds $\ast L_v^a = L_v^r$ and $\ast_{Y_v} L_v^a = (L_v^a)^\perp$ (and similarly with absolute and relative limiting values interchanged). Here, the last identity holds in $\ker D_{Y_v}$.
\end{prop}

%---  mayer-vietoris  --------------------------------------------------------------------------------------
\section{The Generalised Mayer-Vietoris Sequence}\label{sec:gen.mv}

In this section we introduce a generalised Mayer-Vietoris argument which will later enable us to tell whether extended $L^2$ solutions \emph{match at infinity} or not. For the sake of completeness we start with a result from Atiyah, Patodi and Singer \cite[4.9]{aps1}:

\begin{prop}[Atiyah-Patodi-Singer] Let $H^*\big(X_v(\infty)\big)$, respectively $H_0^*\big(X_v(\infty)\big)$, denote de Rham cohomology of $X_v(\infty)$ (with compact supports). Then,
  \[ \ker\D_v(\infty) \cong \im\Big(\, H_0^*\big(X_v(\infty)\big) \rightarrow H^*\big(X_v(\infty)\big)\,\Big) \;. \]
\end{prop}

Now consider $u_a \in L_v^a$. $u_a$ defines a class $[u_a] \in H^*(Y_v)$ and there is at least one extended $L^2$ solution $u$ with limiting value $(u_a,u_r)$ and vanishing $L^2$ part. If we denote by the symbol $\sharp$ the lift of maps to the level of cohomology and by $i_v : Y_v \hookrightarrow \overline{X_v}$ inclusion, we clearly have $i_v^\sharp[u] = [u_a]$, since the relative part $u_r$ pulls back to zero. Hence, $[u_a]$ is the pull-back of a class in $H^*(X_v)$ (more precisely, possibly more than one) and there is a well-defined map
  \begin{equation*}
    \gamma : L^a_v \longrightarrow \im \Big(\, H^*(X_v) \rightarrow H^*(Y_v) \,\Big) \;.
  \end{equation*}

The \emph{Hodge Decomposition Theorem} for closed manifolds implies that $\gamma$ is injective. Now consider the double $\widetilde{X}_v$ of $X_v$ (two copies $X_v^+$, $X_v^-$ of $\overline{X_v}$ with boundaries identified and endowed with opposite orientations). Then, for any $0 \neq \eta \in \im \big( H^*(X_v) \rightarrow H^*(Y_v) \big)$, $\eta = i_v^\sharp \xi$, we may construct a class $\overline{\xi} \in H^*(\widetilde{X}_v)$ which (along the respective inclusion maps) pulls back to $\xi$ on \emph{both} copies of $X_v$ and to $\eta$ on $Y_v$: We need only consider the Mayer-Vietoris sequence for the pair $( X_v^+, X_v^- )$. The unique harmonic representative $\overline{u}$ of $\overline{\xi}$ is thus invariant under the map induced by interchanging corresponding points of the two identical copies. (Let us denote this map by $j^*$ for the moment.) On a small enough neighbourhood $\Sigma$ of $Y_v$ in $\widetilde{X}_v$ this has the following consequences:
  \[ \begin{split}
    0 = \overline{u} - j^*\overline{u} &= 2 u^1 \wedge dt \qquad\text{and}\\
    0 = \overline{u} - j^*\overline{u}
     &= \sum_\lambda \overline{u}_\lambda (e^{\lambda t} - e^{-\lambda t}) \pi^*\phi_\lambda \;,
  \end{split} \]
where we used representations \eqref{normal.form.decomp} and \eqref{forms.expand}. This shows that $\overline{u}\big|_{X_v}$ induces (in the sense of \thref{exts.and.kers}) an extended $L^2$ solution on $X_v(\infty)$ with limiting value $(u^0\big|_{Y_v},0)$. Since by construction $[u^0\big|_{Y_v}] = \eta$, this shows that $\gamma$ is surjective, hence an isomorphism.

\begin{prop}\label{my.coho.iso}There is an isomorphism \; $L_v^a \cong \im \Big(\, H^*(X_v) \rightarrow H^*(Y_v) \,\Big)$.
\end{prop}

Notice the important construction carried out above: For any $u_a \in L_v^a$ there is an extended solution $u \in \ker\D_v^{ex}(\infty)$ so that $u = [u]_{ex}$ and $B(u) = (u_a,0)$. This again reflects the decomposition of $L_v$ as $L_v^a \oplus L_v^r$ and it will be of much use subsequently.

%---  vector-valued cohomology  ----------------------------------------------------------------------------
\bigskip\textbf{Vector-Valued \v{C}ech-de Rham Cohomology.} We will now use the structure of the underlying graph $\mathcal{G}=(V,E)$, in particular its cohomology, in order to encode the adiabatic behaviour of $\ker\D(r)$ by means of a short exact sequence. This will moreover motivate our definition of matching sets of extended harmonic forms. In this section it is important to note that we regard $\mathcal{G}$ as a set of vertices and edges; the realisation of $\mathcal{G}$ as a topological space will be denoted by $|\mathcal{G}|$.

First of all, suppose there are real vector spaces $A_v$, $B_e$ for all $v\in V$, $e\in E$ and linear maps $l_{v,e} : A_v \longrightarrow B_e$ for all half-edges $(v,e)$. We define $A = \bigoplus_v A_v$, $B = \bigoplus_e B_e$ and let $l = \sum l_{v,e} : A \longrightarrow B$ act component-by-component. Further, let $C_*(\mathcal{G})$, $C^*(\mathcal{G})$ be the simplicial (co-)homology groups of $\mathcal{G}$ and $\partial_\mathcal{G}$, $\partial_{\mathcal{G}}^*$ the respective (co-)boundary maps. The spaces
 \[ \begin{array}{ccl}
     C^0(\mathcal{G};A) &= &\big\{\, c_0 : C_0(\mathcal{G}) \longrightarrow A \:\big|\,
       c_0 \text{\, linear and \,} c_0(v) \in A_v \,\big\} \\[3pt]
     C^1(\mathcal{G};A) &= &\big\{\, c_1 : C_1(\mathcal{G}) \longrightarrow B \,\big|\,
       c_1 \text{\, linear and \,} c_1(e) \in B_e \;\big\}
 \end{array} \]
consist of \textbf{$\mathcal{G}$-cochains with values in $A$, $B$} and we want to fit them into a differential complex:

\begin{prop}\label{vector.coho} For $f \in C^0(\mathcal{G};A)$ let $\rho_\mathcal{G} f = l \circ f \circ \partial_\mathcal{G}$, e.g. $(\rho_\mathcal{G}f) (e) = l_{v^\prime,e}\big(f(v^\prime)\big) - l_{v,e} \big(f(v)\big)$. Then,
 \[ 0 \longrightarrow C^0(\mathcal{G};A) \xrightarrow{\;\rho_\mathcal{G}\;} C^1(\mathcal{G};B) \longrightarrow 0 \]
is a differential complex with \textbf{vector-valued cohomology groups} $H^0(\mathcal{G};A) = \ker \rho_\mathcal{G}$ and $H^1(\mathcal{G};B) = \coker \rho_\mathcal{G}$.
\end{prop}

Now, we turn to specific choices of $A_v$, $B_e$ and $l_{v,e}$. Let $X^0$ be the disjoint union of the elements of an open cover of $X(r)$ consisting of slightly prolonged open vertex manifolds $\widehat{X}_v$. Then, for each $v \in V$ there is exactly one element of the cover containing $X_v(r)$; if $e$ connects $v$ with $v^\prime$ in $\mathcal{G}$, $Y_e$ is a deformation retract of $\widehat{X}_v \cap \widehat{X}_{v^\prime} = \widehat{Y}_e$; $\widehat{X}_v$, $\widehat{X}_{v^\prime}$ are disjoint in $X(r)$ if $v$ and $v^\prime$ are not adjacent in $\mathcal{G}$. Denote the disjoint union of the $\widehat{Y}_e$'s by $X^1$ and the exterior algebras of $X^0$, $X^1$ by $F^0$ respectively $F^1$, e.g.
 \[ C^\infty(F^0) = \bigoplus_{v \in V} C^\infty( \Lambda T^*\widehat{X}_v ) \;. \]

As described for instance in \cite[\S\S 1-2]{bt82}, the inclusions $X^1 \rightrightarrows X^0 \rightarrow X$ induce a short exact sequence of differential complexes
  \begin{equation}\label{F.complex}
    0 \longrightarrow C^\infty(F) \xrightarrow{\;\tau\;} C^\infty(F^0) \xrightarrow{\;\rho\;} C^\infty(F^1)
       \longrightarrow 0 \;,
  \end{equation}
where for $u \in C^\infty(F)$ we have $\tau u = \sum_{v\in V} u\big|_{\widehat{X_v}}$ and for $u \in C^\infty(F^0)$:
  \[ \rho u = \rho\left(\sum_{v\in V}u_v\right)
     = \sum_{e=(v,v^\prime)} u_{v^\prime}\big|_{\widehat{Y_e}} - u_{v}\big|_{\widehat{Y_e}} \]
\eqref{F.complex} is the \textbf{\v{C}ech-de Rham complex} for the cover $X^0$ of $X$. Since $\tau$ and $\rho$ commute with exterior differentiation this leads to a long exact sequence in cohomology
  \begin{equation}\label{gen.mv}
    \dotsm \xrightarrow{\;\rho^\sharp\;} H^{k-1}(X^1) \xrightarrow{\;b^\sharp\;} H^k(X)
    \xrightarrow{\;\tau^\sharp\;} H^k(X^0) \xrightarrow{\;\rho^\sharp\;} H^k(X^1) \longrightarrow \dotsm
  \end{equation}
to which we refer to as the \textbf{generalised Mayer-Vietoris sequence} of $X$ with respect to $\mathcal{G}$.

\bigskip The combinatorial structure of the cover $X^0$ may be encoded by a map
  \[ \mathcal{X} \;:\; X(r) \longrightarrow \mathcal{G} \;,\quad x \longmapsto \Bigg\{
     \begin{array}{cl} v\,, &x \in \overline{X_v} \\ e\,, &x = (y,t) \in Z_e(r), \;|t| < r \end{array} \]
which is onto and whose preimages cover the base $X(r)$. (This map motivates the notion of \textbf{manifolds fibred over graphs}.) The set of \emph{fibres} (i.e. preimages) of $\mathcal{X}$ is in one-to-one correspondence with $(X^0,X^1)$, each fibre being a deformation retract of a unique element of $X^0$ or $X^1$.

Since we defined $\mathcal{X}$ to map into $\mathcal{G}$ rather than $|\mathcal{G}|$, we cannot regard $\mathcal{X}$ as an actual fibration; even if we used $|\mathcal{G}|$, we could not lift arbitrary paths let alone homotopies. Nevertheless, we continue referring to the fibres of $\mathcal{X}$.

Using \thref{vector.coho}, we may efficiently rewrite sequence \eqref{gen.mv}, connecting it to the cohomology of $\mathcal{G}$. To do so, let $A_v = H^*(\widehat{X}_v)$, $B_e = H^*(\widehat{Y}_e)$ and $l_{v,e} : H^*(\widehat{X}_v) \longrightarrow H^*(\widehat{Y}_e)$ be induced by inclusion. Since the map $\mathcal{X}$ encodes not only the spaces $A$ and $B$ but also the graph $\mathcal{G}$ we replace any reference to $\mathcal{G}$ in vector-valued cohomology by reference to $\mathcal{X}$.

The defining condition for the spaces of $\mathcal{X}$-cochains $C^j(\mathcal{X})$ implies that $C^j(\mathcal{X}) \cong H^*(X^j)$, since it forbids interactions between objects defined on different vertices or edges. (This is of course essential for reformulating the Mayer-Vietoris sequence.) Moreover, the strong connection between $\rho_{\mathcal{X}} = \rho_{\mathcal{G}}$ and $\rho^\sharp$ should be obvious. In fact, using the isomorphisms
 \[ \hat{\imath}_j : H^*(X^j) \longrightarrow C^j(\mathcal{X}) \;,\;
    \hat{\imath}_j u = \left[ \;\sum_p c_p p \longmapsto \sum_p c_p u_p \;\right] \;, \]
with summation over the set of $j$-simplices $p$, we easily obtain $\hat{\imath}_1 \circ \rho^\sharp = \rho_{\mathcal{X}} \circ \hat{\imath}_0$. If we denote the cohomology groups $H^j\big(\mathcal{G};H^*(X^j)\big)$ by $H^j(\mathcal{X})$ for short, this in turn shows that there are induced isomorphisms
  \begin{equation}\label{ihat.isos}\begin{array}{ccccl}
    \hat{\imath}_0^\sharp &: &\im\tau^\sharp &\longrightarrow &H^0(\mathcal{X}) \\[3pt]
    \hat{\imath}_1^\sharp &: &\coker\rho^\sharp &\longrightarrow &H^1(\mathcal{X})
  \end{array} \end{equation}
and combining this with sequence \eqref{gen.mv}, we obtain:

\begin{prop} The following short sequence is exact:
 \begin{equation}\label{combined.ses}
   0 \xrightarrow{\qquad} H^1(\mathcal{X}) \xrightarrow{\; b^\sharp \circ \hat{\imath}_1^\sharp } H^*(X(r))
     \xrightarrow{\;\; \hat{\imath}_0^\sharp \circ \,\tau^\sharp \;} H^0(\mathcal{X}) \xrightarrow{\qquad} 0
  \end{equation}
\end{prop}
We ought to remark that in the second step of \eqref{combined.ses}, the degree is raised by $1$, i.e. a class $\alpha$ having values in $H^k(X^1)$ is mapped to a form on $X(r)$ of degree $k+1$.

%---  matching extended forms  -----------------------------------------------------------------------------
\section{Matching Extended Harmonic Forms}\label{sec:match.ext}

Taking into account sequence \eqref{F.complex}, a set of extended harmonic forms $u = \sum_v u_v \in \bigoplus_v \ker\D_v^{ex}(\infty)$ should be called a matching set if it is an element of $\im \tau$. Since extended forms are defined on the stars $X_v(\infty)$ rather than being defined on $X^0$, this obviously does not make sense. But $\im \tau = \ker \rho$ and the map $\rho$ can easily be adapted to this situation.

Let $X(\infty)$ be the disjoint union of the stars $X_v(\infty)$, $F(\infty)$ be the corresponding exterior algebra and $\D^{ex}(\infty) : H^1_{ex}\big(F(\infty)\big) \rightarrow L^2_{ex}\big(F(\infty)\big)$ be defined component-by-component. Additionally, consider the \textbf{total boundary data map}
  \[ B : \ker\D^{ex}(\infty) \longrightarrow C^\infty(\widehat{F}) \oplus C^\infty(\widehat{F}) \quad,\quad
         \sum_v u_v \longmapsto \sum_{e=(v,v^\prime)} \big( B_{(v,e)}(u_v), B_{(v^\prime,e)} (u_{v^\prime}) \big) \]
and observe that its range consists of two copies of boundary data since data is given for each \emph{half-edge}. (This corresponds to incoming and outgoing limiting values, the choice depending on the orientation of $\mathcal{G}$.)

\begin{define} Let
  \begin{equation*}\begin{array}{rccl}
    \rho_{ex} : &\ker\D^{ex}(\infty) &\longrightarrow &C^\infty(\widehat{F}) \\[3pt]
                &\sum_v u_v          &\longmapsto
                     &\displaystyle\sum_{e=(v,v^\prime)} B_{(v^\prime,e)}(u_{v^\prime}) - B_{(v,e)}(u_v) \;,
  \end{array}\end{equation*}
where $B_{(v,e)}$ denotes the boundary data map corresponding to the half-edge $(v,e)$. We define the space of \textbf{matching sets of extended solutions} $\mathcal{W}$ and the space of \textbf{limiting values of matching sets} $\mathcal{L}$ by
  \begin{equation*}\begin{split}
    \mathcal{W} &:= \ker \rho_{ex} \subset \ker\D^{ex}(\infty) \\
    \mathcal{L} &:= \im \Big( \pi_1 : B(\mathcal{W}) \longrightarrow C^\infty(\widehat{F}) \Big) \;,
  \end{split}\end{equation*}
with the map $\pi_1$ defining $\mathcal{L}$ being projection onto the first factor (or equivalently, the second).
\end{define}

Thus, we call a set of extended harmonic forms a matching set if the asymptotics -- or limiting values -- on cylindrical parts corresponding to the same edge coincide. With regard to the decomposition $\ker\D_v^{ex}(\infty) \cong \ker\D_v(\infty) \oplus L_v$ it should be clear that a similar decomposition holds for matching sets as well. But this time, we may only take into account limiting values of matching sets.

\begin{prop} With $i$ denoting inclusion, the sequence
  \[ 0 \longrightarrow \ker\D(\infty) \xrightarrow{\quad i \quad} \mathcal{W}
       \xrightarrow{\;\pi_1 \circ B\;} \mathcal{L} \longrightarrow 0               \]
is exact. Moreover, it holds
  \begin{equation}\label{matching.limits.decomp}\begin{split}
    \mathcal{W} &\cong \ker\D(\infty) \oplus \mathcal{L} \\
          &\cong \ker\D(\infty) \oplus \mathcal{L}^a \oplus \mathcal{L}^r \;,
  \end{split}\end{equation}
in which $\mathcal{L}^a = (\pi_1 \circ B^a)(\mathcal{W})$ and $\mathcal{L}^r=(\pi_1 \circ B^r)(\mathcal{W})$ can be identified with subspaces of $\mathcal{W}$ consisting of matching sets with vanishing relative respectively absolute limiting values.
\end{prop}

\begin{proof}It is evident that the sequence is exact, we need only prove \eqref{matching.limits.decomp}. Mapping to the $L^2$ part $[\,\cdot\,]_{L^2}$ is a left inverse for $i$, and since the limiting values uniquely determine the extended parts of matching sets, mapping to the corresponding extended part $[\,\cdot\,]_{ex}$ is a right inverse for $\pi_1 \circ B$. Then it is clear that
  \[ [\,\cdot\,]_2 \oplus \pi_1 \circ B : \mathcal{W} \longrightarrow \ker \D(\infty) \oplus \mathcal{L} \]
defines an isomorphism. (In fact, this map defines a splitting for the above sequence.)

Now let $w \in \mathcal{W}$ have limiting value $(w_a, w_r)$. Using the remark after \thref{my.coho.iso}, there is $\overline{w} \in \ker\D^{ex}(\infty)$ with limiting value $(w_a, 0)$ satisfying $[\overline{w}]_{L^2} = 0$. Defining $\widetilde{w} = w - \overline{w} - [w]_{L^2}$, we obtain two matching sets $\overline{w}$ and $\widetilde{w}$ with limiting values $(w_a,0)$ and $(0, w_r)$ so that $w = \overline{w} + \widetilde{w}$. Since these forms are uniquely determined by $w$, this gives \eqref{matching.limits.decomp}.
\qed\end{proof}

%---  matching and cohomology  -----------------------------------------------------------------------------
\bigskip\textbf{Matching Sets and Cohomology.} As a consequence of the construction used in the proof of \thref{my.coho.iso}, we are now able to relate subspaces of matching sets to the kernel and the image of the map $\rho^\sharp$. In the end, these considerations will allow us to make a connection between sequence \eqref{combined.ses} and the adiabatic limit of $\ker\D(r)$ as $r \rightarrow \infty$.

\begin{prop}\label{matching.vs.coho.1}There is an isomorphism of vector spaces
  \[ \alpha : \ker \D(\infty) \oplus \mathcal{L}^a  \longrightarrow  \ker \rho^\sharp \;, \]
hence equality of dimensions $\dim \ker \D(\infty) + \dim \mathcal{L}^a = \dim H^0(\mathcal{X})$ holds.
\end{prop}

\begin{proof}To begin with, we decompose $u \in \ker \D(\infty) \oplus \mathcal{L}^a$ as $u = [u]_{L^2} + [u]_{ex}$. Since $[u]_{L_2}$ is $L^2$ harmonic, it is closed and defines a class $\xi = \sum_v \xi_v \in H^*(X^0)$. Notice that $i_v^\sharp \xi = [ \xi_v\big|_{Y_v} ] = 0$.

Now, since $u_a$ is closed and coclosed as a form on $Y$ and does not contain a factor $dt$, using a \emph{Green-Stokes Formula} (as for instance in \cite[p.160]{taylor1}) for the restrictions to the disjoint union $X^0(s) = \bigcup_v X_v(s)$, $s>0$, we can show that on any of the compact manifolds $X^0(s)$ we have $\big\|d[u]_{ex}\big\|_{L^2} = \big\|\delta [u]_{ex}\big\|_{L^2}$ and $d[u]_{ex} \perp \delta [u]_{ex}$. Since $\D^{ex}(\infty) [u]_{ex} = 0$, this shows
  \[ 0 = \big\| \D^{ex}(\infty) u\big|_{X^0(s)} \big\|_{L^2}^2 = \big\| du\big|_{X^0(s)} \big\|_{L^2}^2 + \big\| \delta u\big|_{X^0(s)} \big\|_{L^2}^2
       = 2 \big\| du\big|_{X^0(s)} \big\|_{L^2}^2 \]
and by continuity we conclude that $[u]_{ex}$ is closed as a form in $C^\infty\big(F^0(\infty)\big)$, hence defines a class $\eta \in H^*(X^0)$ which for $e = (v,v^\prime)$ satisfies $i^\sharp_{(v,e)}\eta_v = i^\sharp_{(v^\prime,e)} \eta_{v^\prime}$.

These arguments show that $u$ maps to a class $\alpha(u) = \xi + \eta \in H^*(X^0)$ and considering the way in which $\rho^\sharp$ acts we clearly obtain $\rho^\sharp \alpha(u) = 0$.

\bigskip Now suppose $\alpha(u) = 0$. Then $u$ is exact, but since $u$ has vanishing relative limiting value this implies $u_a = 0$ as well. Being harmonic and exact, the $L^2$ Hodge decomposition shows $u = 0$ and thus injectivity of $\alpha$. Using the construction referred to after \thref{my.coho.iso}, surjectivity of $\alpha$ follows immediately.
\qed\end{proof}

\begin{prop}\label{matching.vs.coho.2}When considered as subspaces of $H^*(Y)$, the inclusion $\mathcal{L}^r \subset \big( \im\rho^\sharp \big)^\perp$ holds, hence we have $\dim \mathcal{L}^r \le \dim H^1(\mathcal{X})$.
\end{prop}

\begin{proof}Let $w \in C^\infty(F^0)$ and $u \in \mathcal{L}^r$. Considering $u$ as an element of $H^*(Y)$ (thus forgetting the absolute component which equals $0$), by \thref{limiting.values.decomp} there is a unique form $\overline{u} \in \ker\D^{ex}(\infty)$ with limiting value $(\ast_Y u, 0)$. Now let $g : X^0 \rightarrow \R$ be any smooth function so that near $Y_e \subset \partial\widehat{X}_v$ we have $g = \bullet_{(v,e)}$ and let $i : Y \cup Y = \partial X^0 \hookrightarrow X^0$ denote inclusion. Then,
  \[ \rho w = \sum_v \Big( \sum_{e \sim v} \bullet_{(v,e)} w\big|_{Y_{(v,e)}} \Big) = i^* \big( g w \big) \]
which shows that $\rho w$ is the pull-back of a smooth form $gw \in C^\infty(F^0)$. Thus, we obtain
  \begin{equation*}\begin{split}
   \inpro{ \rho w }{ u }_{H^*(Y)} &= \int_Y \rho w \wedge \ast_Y u = \int_Y i^* \big( gw \wedge \overline{u} \big) \\
     &= \int_Y i^* \big( w \wedge g \overline{u} \big) = \int_Y (i^*w) \wedge \rho\overline{u} = 0 \;,
  \end{split}\end{equation*}
since $\overline{u} \in \ker\rho$. This shows $\mathcal{L}^r \subset \big( \im\rho^\sharp \big)^\perp \cong H^1(\mathcal{X})$.
\qed\end{proof}

%--  spectral sequences  -----------------------------------------------------------------------------------
\bigskip\textbf{Spectral Sequences.} There is another way of looking at the spaces $H^*(\mathcal{X})$ and sequence \eqref{combined.ses}. In \cite[pp. 166]{bt82} a spectral sequence for the \v{C}ech-de Rham complex is studied. A \textbf{spectral sequence} is a sequence of differential groups $(E_k,d_k)$ in which each group is the homology of its predecessor. In this case, we let $C^{p,q}(\mathcal{X})$ be the space of $\mathcal{X}$\hspace{-2pt}-cochains of degree $p$ with values in $H^q\big(\mathcal{X}^{-1}(\mathcal{G})\big)$ (since this cohomology is a direct sum with respect to $q$, any $\mathcal{X}$\hspace{-2pt}-cochain may decomposed into parts having values in the respective degree) and define
 \begin{alignat*}{3}
   E_1 &= \displaystyle\bigoplus_{p,q} E_1^{p,q}, &\quad&\text{where} &\quad
    &E_1^{p,q} = C^{p,q}(\mathcal{X}) .\\
   \intertext{Taking the homology with respect to $\rho_{\mathcal{X}}$ successively leads to the next term:}
   E_2 &= \displaystyle\bigoplus_{p,q} E_2^{p,q}, &&\text{where}
    &&E_2^{p,q} = H_{\rho_{\mathcal{X}}}E_1^{p,q} =
         \left\{\begin{array}{lcc} \ker\rho_{\mathcal{X}}\big|_{C^{0,q}(\mathcal{X})} &, &p=0 \\[4pt]
                                 \coker\rho_{\mathcal{X}}\big|_{C^{1,q}(\mathcal{X})} &, &p=1 \end{array}\right.
 \end{alignat*}
and $E_k = E_2$ for $k > 2$. Then, on the one hand it is plain to see that
  \[ H^p(\mathcal{X}) = \displaystyle\bigoplus_q E_2^{p,q} \;, \]
and on the other hand, since this spectral sequence converges to the (de Rham) cohomology of $X(r)$, we yet again obtain
  \[ H^*\big(X(r)\big) \cong E_2 = H^0(\mathcal{X}) \oplus H^1(\mathcal{X}) \;. \]
This may as well be compared to Mazzeo and Melrose's \textbf{Hodge-Leray spectral sequence} (cf. \cite{mm90}). In their work, a fibration $M \rightarrow N$ with typical fibre $F$ ($M$ and $F$ being compact manifolds) is considered and using a parameter $x$ the Riemannian manifold $N$ is stretched as $x \rightarrow 0$ and the adiabatic behaviour of the spaces of harmonic forms on this stretched manifold is studied. Mazzeo and Melrose analytically define Hodge theoretic analogues $E_N^k$ of the terms $E_k$ and show that this sequence converges to a space $E_\infty^k$ which satisfies
  \begin{equation}\label{mm.iso} E_\infty^k \otimes \mathbb{L} \cong H^k(M)
    \otimes \mathbb{L} \qquad\text{(as $\mathbb{L}$-modules)} \;,
  \end{equation}
where $\mathbb{L}$ denotes the ring of formal Laurent series in $x$. In fact, Mazzeo and Melrose's terms $E^k_N$ correspond to the terms $\oplus_{q-p=k} E^{p,q}_N$ for $N=1,2$ as defined above. The important difference being that in our case the base is one dimensional, hence the sequence $E_j$ becomes stationary in the second step and that we actually don't have a fibration. However, we rely heavily on certain assumptions on the geometry of $X$ (e.g. the product structure near cylindrical parts) -- these are not made in \cite{mm90}.

The isomorphism \eqref{mm.iso} is of course a more general result than what we aim at, since the base $N$ may be any smooth Riemannian manifold (in particular of dimension $n > 1$), but our result includes the case of the base being singular in a certain way (i.e. being a graph).

%---  splicing  --------------------------------------------------------------------------------------------
\section{The Splicing Construction}\label{sec:splicing}

Now we will define a variant of the \textbf{Cappell-Lee-Miller splicing map}, adapted to the situation of multiple edges. It is a linear map $\mathcal{S}_r : \mathcal{W} \longrightarrow C^\infty\big(F(r)\big)$, gluing a set of matching solutions to form an approximately harmonic form on $X(r)$. We will show that, if composed with a suitable spectral projection $\Pi_r$, splicing defines an isomorphism $\mathcal{W} \cong \ker D(r)$. Our approach follows the lines of \cite{clm1}, but is also influenced by the work of Nicolaescu \cite{nicolaescu02}.

To begin with, we recall that even though $X_v(r)$ can be considered a submanifold of $X_v(\infty)$, the normal coordinate $t$ on the cylindrical parts has a different range. In order to keep track of the range of the normal coordinate for restrictions of the type
  \[ u\big|_{X_v(r)} \quad\text{where}\quad u \in L^2_{ex}\big(F(\infty)\big) \]
we define a new coordinate $\vartheta = \bullet t + r$. (See \eqref{bullet} for the definition of $\bullet$.) A change of variables from $\vartheta$ to $t$ coordinates will be expressed by a pull-back $\vartheta^*$ (and often used implicitly). Also, we will make use of a cut-off function
  \[ g_r : \R_+ \longrightarrow [0,1] \quad,\quad
         g_r(\vartheta) = \left\{\begin{array}{ccl} 1 &, &\vartheta \le r-\frac{3}{4} \\
                                                 0 &, &\vartheta \ge r-\frac{1}{4} \end{array}\right. \]
with $|\partial_\vartheta g_r|\le 4$, extended to $X(\infty)$ in the obvious way.

\begin{define} Let $w=(w_v) \in \mathcal{W}$ be a matching set of harmonic forms with limiting value $\widehat{w}=(\widehat{w}_e) \in \mathcal{L}$. We define smooth forms by
  \begin{equation*}\begin{array}{lcl}
    S_r^v(w) \in C^\infty\big(F_v(r)\big) &, &S_r^v(w) = \vartheta^*\big( g_r \cdot w_v \big) \\[3pt]
    S_r^e(w) \in C^\infty\big(\pi^*\widehat{F}_e\big) &, &S_r^e(w) = \big( 1 - \vartheta^*g_r )
     \pi^*\widehat{w}_e \quad.
  \end{array}\end{equation*}
Then, the \textbf{Cappell-Lee-Miller splicing map for graphs} is defined by
  \[ \mathcal{S}_r : \mathcal{W} \longrightarrow C^\infty\big(F(r)) \quad,\quad
     \mathcal{S}_r(w) = \sum_{v \in V} S_r^v(w) + \sum_{e \in E} S_r^e(w) \;, \]
where the single summands are assumed to vanish outside their respective domain, e.g. $S_r^v(w)=0$ on $X_{v^\prime}(r)$, $v\neq v^\prime$.
\end{define}

% ---  projected splicing  ---------------------------------------------------------------------------------
\bigskip\textbf{On Projected Splicing.} Let $\spa(r,R)$ be the subspace of $L^2\big(F(r)\big)$ spanned by eigenforms of $\D(r)$ corresponding to eigenvalues $|\mu| \le R$. Throughout the rest of this article we will denote the smallest spectral gap of the operators $\widehat{D}_e$ by $\lambda_0$, i.e.
  \[ 0 < \lambda_0 := \min_{e \in E} \lambda_e , \quad\text{where}\quad
     \lambda_e := \inf\big\{ \; |\lambda| \; \big| \; 0 \neq \lambda \in \spec(\widehat{D}_e) \; \big\} \;. \]
Even though splicing does not yield harmonic forms nor eigenforms (since $\D(r)$ also acts on $\vartheta^*g_r$), it is true (cf. \cite[4.1]{clm1}) that spliced forms are exponentially close to linear combinations of eigenforms corresponding to very small eigenvalues. (We refer to eigenvalues of $\D(r)$ tending to $0$ exponentially in $r$ as \textbf{very small eigenvalues}.)

\begin{prop}[Cappell-Lee-Miller]\label{splicing.is.mono}Let $\Pi_r$ be orthogonal projection of $L^2\big(F(r)\big)$ onto $\spa\big(r,e^{-\frac{\lambda_0 r}{4}}\big)$. Then, there is $R \ge 2$ such that for all $r \ge R$, the \textbf{projected splicing map}
  \begin{equation*}
    \Pi_r \circ \mathcal{S}_r : \mathcal{W} \longrightarrow \spa\big(r,e^{-\frac{\lambda_0 r}{4}}\big)
  \end{equation*}
is injective. Moreover, for $w\in \mathcal{W}$ we have
  \begin{equation*}
    \big\| \big( \Pi_r \circ \mathcal{S}_r - \mathcal{S}_r \big)(w) \big\|_{L^2}
      \le e^{-\frac{\lambda_0 r}{4}} \big\| \mathcal{S}_r(w) \big\|_{L^2} \;.
  \end{equation*}
\end{prop}

% ---  exact main result  ----------------------------------------------------------------------------------
\section{Exact Statement of the Main Result}\label{sec:main}

Having introduced matching sets of harmonic forms and the splicing construction, we formulate the theorem we aim at in an exact manner.

\begin{theorem}\label{main.exact} Let $\varepsilon > 0$ and $X$ be a manifold fibred over a graph $\mathcal{G}$. Then, there is $r_0>0$ such that for $r > r_0$:
 \begin{enumerate}
  \item There are no non-zero eigenvalues $\mu$ of $D(r)$ satisfying $|\mu| < r^{-(1+\varepsilon)}$.
  \item With $\Pi_{\ker}$ denoting the orthogonal projection onto $\ker \D(r)$, the projected splicing map
          \[ \Pi_{\ker} \circ \mathcal{S}_r : \mathcal{W} \longrightarrow \ker \D(r) \]
        is an isomorphism of matching sets of extended harmonic forms on $X(\infty)$ and harmonic forms on $X(r)$.
  \item Any spliced form is exponentially close to its projection onto the space of harmonic forms,
          \[ \big\| \big( \Pi_{\ker} \circ \mathcal{S}_r - \mathcal{S}_r \big) (u) \big\|_{L^2}
               \le e^{-\frac{\lambda_0 r}{4}} \big\| \mathcal{S}_r (u) \big\|_{L^2}  \;.\]
  \item The topological representation of $H^*(X)$ by means of a \v{C}ech-de Rham complex gives the asymptotics of the
        Hodge cohomology of $X(r)$ as $r \rightarrow \infty$:
          \[ \begin{diagram}
             \node{0} \arrow{e} \node{H^1(\mathcal{X})} \arrow{e} \node{H^*(X)}
              \arrow{e} \node{H^0(\mathcal{X})} \arrow{e} \node{0} \\
             \node{0} \arrow{e} \node{\mathcal{L}^r} \arrow{e}\arrow{n,l}{\hat{\imath}} \node{\mathcal{W}}
              \arrow{e} \arrow{n} \node{\ker\D(\infty)\oplus\mathcal{L}^a} \arrow{e}\arrow{n,r}{\alpha} \node{0}
          \end{diagram} \]
        Here, the vertical arrows represent isomorphisms and the rows are exact.
 \end{enumerate}
\end{theorem}

\begin{proof} The most difficult part of this proof is showing that splicing defines a surjective map as claimed in \emph{ii)}. The convergence results of section \ref{sec:convergence} will be crucial for this, and the proof of surjectivity is postponed until section \ref{sec:surjective}.

\thref{splicing.is.epi} shows that for $r$ sufficiently large, the number $h(r)$ of non-zero eigenvalues (counted with multiplicity) $\mu$ of $D(r)$ satisfying $|\mu| < r^{-(1+\varepsilon)}$ is less than or equal to $\dim\mathcal{W}$ whereas Propositions \ref{matching.vs.coho.1} and \ref{matching.vs.coho.2} have shown that $\dim\mathcal{W} \le \dim\ker D(r)$. Since clearly $\dim\ker D(r) \le h(r)$, this establishes item \emph{i)}. But then, the projections $\Pi_{\ker}$ and $\Pi_r$ coincide, and item \emph{iii)} follows from \thref{splicing.is.mono}. Moreover, this shows that
  \[ \Pi_{\ker} \circ \mathcal{S}_r \,:\, \mathcal{W} \longrightarrow \ker D(r) = \ker \D(r) \]
is an isomorphism. This also implies that the isomorphism $\hat{\imath}_1^\sharp : \coker \rho^\sharp \longrightarrow H^1(\mathcal{X})$ of \eqref{ihat.isos} induces an isomorphism
  \[ \hat{\imath} : \mathcal{L}^r \longrightarrow H^1(\mathcal{X}) \quad,\quad u \longmapsto \big[B^r(u)\big] \;. \]
Along with the results of section \ref{sec:match.ext}, we obtain the diagram of item \emph{iv)} and the proof of the main result is reduced to the proof of \thref{splicing.is.epi}.
\qed\end{proof}

%---  convergence results  ---------------------------------------------------------------------------------
\section{Convergence Results}\label{sec:convergence}

The main step in showing that projected splicing defines an isomorphism is to show that this map is onto. But before we may address this, we need to collect some convergence results taken from \cite{clm1}. In the following we will always denote eigenvalues of $\widehat{D}$ by $\lambda$ and those of $\D(r)$ by $\mu$.

\begin{prop}[Cappell-Lee-Miller]\label{low.mode.expansion} For $|\mu| < \lambda_0$, an eigenform $u$ on the cylindrical manifold $Z(r+1)$ corresponding to the eigenvalue $\mu$ can be written as a sum $u = u_0 + u_+ + u_-$ where the summands have the following representations:
 \[ u_0(t,\cdot\,) = \sum_{1\le k\le N}
           a_k\big(\cos(\mu t) \pi^*\phi_k + \sin(\mu t) \pi^*\phi_{-k}\big)
           + b_k\big(\cos(\mu t) \pi^*\phi_{-k} - \sin(\mu t) \pi^*\phi_k\big) \;, \]
with constants $a_k$ and $b_k$,
 \[ u_+(t,\cdot\,) = \sum_{k>N} c_k^+ e^{\alpha_k t} \pi^*\phi_k^+ \qquad\text{and}\qquad
    u_-(t,\cdot\,) = \sum_{k>N} c_k^- e^{-\alpha_k t} \pi^*\phi_k^- \;, \]
with constants $c_k^+$, $c_k^-$ and $\alpha_k = \sqrt{\lambda_k^2 - \mu^2}$ depending on $r$ and
 \[ \phi_k^+ = (\alpha_k + \lambda_k)\phi_k + \mu \phi_{-k} \qquad\text{and}\qquad
    \phi_k^- = \mu \phi_k + (\alpha_k + \lambda_k)\phi_{-k} \;. \]

Further, every form of this type is a solution to $D(r+1)u = \mu u$ on $Z(r+1)$. If $\mu = 0$, we have $\phi_k^{\pm} = 2\lambda_k \phi_{\pm k}$.
\end{prop}

Using these explicit representations of eigenforms, we are able to show a slightly adapted version of a result from Cappell, Lee and Miller (cf. \cite[6.2]{clm1}). In order to formulate it, the following definition will help.

\begin{define} For any $s>0$, we denote by $\mathcal{R}_s$ the map which restricts forms to the disjoint union of manifolds with boundary $X^0(s)$. For instance, for $r>s$,
  \[\begin{array}{ccccl}
   \mathcal{R}_s &: &L^2\big(F(r)\big) &\longrightarrow &L^2\big( F^0(s)\big) \qquad\text{or} \\[4pt]
   \mathcal{R}_s &: &L^2_{ex}\big(F(\infty)\big) &\longrightarrow &L^2\big( F^0(s) \big) \;.
  \end{array}\]
\end{define}

Let us remark that due to the representations of \thref{exts.and.kers} and \thref{low.mode.expansion}, $\mathcal{R}_s$ is injective when acting on $\mathcal{W}$ respectively on the span of eigenforms of $\D(r)$ corresponding to \textbf{small eigenvalues} (i.e. those tending to $0$ faster than $r^{-1}$).

\begin{thm}\label{local.convergence} Let $r_j$, $j \in \N$, be such that $\lim r_j = \infty$ and $u_j$ be a sequence of smooth forms on $X(r_j)$ satisfying:
 \begin{enumerate}
  \item $D(r_j)u_j = \mu_j u_j$
  \item $|\mu_j| < \lambda_0$ and $\lim \mu_j r_j = 0$
  \item $\mathcal{R}_0 u_j$ is a bounded sequence in $L^2\big(F^0(0)\big)$
 \end{enumerate}
Then, there is a subsequence $(u_{j^\prime})$ of $(u_j)$ and a matching set $w \in \mathcal{W}$ such that for any $s>0$
 \[ \mathcal{R}_s u_{j^\prime}\longrightarrow\mathcal{R}_s w \qquad\text{in}\qquad L^2\big(F^0(s)\big) \;.\]
\end{thm}

\begin{proof}First of all, as already shown in \cite[6.1-2]{clm1}, there is a subsequence $(u_{j^\prime})$ so that
  \[ u_{j^\prime}\big|_{X^0(0)} \longrightarrow w\big|_{X^0(0)} \quad\text{in}\quad H^l\big(F^0(0)\big) \]
for all $l \in \N$. Since restriction to the boundary is continuous as a mapping $H^l \rightarrow H^{l-\frac{1}{2}}$, this gives convergence
  \[ u_{j^\prime}\big|_{\partial X^0(0)} \longrightarrow w\big|_{\partial X^0(0)}
     \quad\text{in}\quad H^l(\widehat{F}) \]
and thus convergence of coefficient functions
  \[\begin{array}{ccl}
     c_k^\pm (j^\prime) \longrightarrow \frac{f_{\pm k}}{2 \lambda_k} &\text{for} &k > N \quad\text{and} \\[3pt]
     a_k(j^\prime) \rightarrow d_k \quad\text{,}\quad b_k(j^\prime) \rightarrow d_{-k} &\text{for}
      &1 \le k \le N \;,
  \end{array} \]
where, with regard to \thref{low.mode.expansion}, $d_k$, $d_{-k}$ denote the coefficients of $w_0$ and $f_{\pm k}$ those of $w_{\pm}$. (To be precise, we ought to equip the coefficient functions with an additional subscript indicating the edge, for simplicity this is omitted. Besides, depending on the orientation of the edge, either $f_k$ or $f_{-k}$ vanishes on that edge.) Using this, direct calculations show that on cylindrical parts for $r_{j^\prime} > s$ we have
  \begin{equation}\label{convergence.of.modes}\begin{array}{ccl}
    \big\| (u_{j^\prime})_0 - w_0 \big|_{Z^0(s)} \big\|_{L^2} &\longrightarrow &0 \quad\text{and} \\[5pt]
    \big\| (u_{j^\prime})_\pm - w_\pm \big|_{Z^0(s)} \big\|_{L^2} &\longrightarrow &0
      \quad\text{as}\quad j^\prime \rightarrow \infty \;,
  \end{array} \end{equation}
where $Z^0(s)$ denotes the disjoint union of cylinders $Z_e(s)$, $e \in E$. We will shortly sketch the calculations involved in showing \eqref{convergence.of.modes}.

Consider a cylinder $Z_{(v,e)}(r_{j^\prime})$ whose boundary is given by $t=-r_{j^\prime}$ and $t=0$. The case of the opposite orientation is treated in the same way. Then, $w_+=0$ on this cylinder because an extended harmonic form possesses no exponentially increasing part. Thus, using an estimate shown by Cappell, Lee and Miller (see \cite[5.4]{clm1}), we obtain
  \[ \big\| (u_{j^\prime})_+ - w_+ \big|_{Z_{(v,e)}(s)} \big\|_{L^2}^2
     <  -\frac{1}{\alpha} \big( 1 - e^{2\alpha s} \big) e^{-4\alpha r_{j^\prime}} \;, \]
which tends to $0$ as $j^\prime \rightarrow \infty$. Concerning the constant part $(u_{j^\prime})_0$ we have
  \[ \big\| (u_{j^\prime})_0 - w_0 \big|_{Z_{(v,e)}(s)} \big\|_{L^2}^2
     \le \,s \sum_{1 \le k \le N} \big( a_k(j^\prime) - d_k \big)^2 + \big( b_k(j^\prime) - d_{-k} \big)^2 \;. \]
Since $ a_k(j^\prime) \rightarrow d_k$ and $b_k(j^\prime) \rightarrow d_{-k}$, this term vanishes as well. As for the decreasing part, using the coordinate $\vartheta$ (defined at the beginning of this section) and \thref{low.mode.expansion}, we may write
  \begin{equation*}\begin{array}{rcll}
   \vartheta^*u_j\big|_{Z^0(s)}(\vartheta) &= &\vartheta^*(u_j)_0(\vartheta) \;+ &\hspace{-5pt}\sum\limits_{k>N}
        \big(c_k^+(j)e^{-\alpha_k(j)r_j}\big)\,e^{\alpha_k(j)\vartheta} \pi^*\phi_k^+ \\
        & & &\quad + \big(c_k^-(j)e^{\alpha_k(j)r_j}\big)\,e^{-\alpha_k(j)\vartheta}
             \pi^*\phi_k^- \;,\\[8pt]
   \vartheta^*w\big|_{Z^0(s)}(\vartheta) &= &\vartheta^*w_0(\vartheta) \;+
        &\hspace{-5pt}\sum\limits_{k > N} \big( f_k e^{-\lambda_k r_j}\big)\,
        e^{\lambda_k \vartheta} \pi^*\phi_k^\bullet \;,
  \end{array} \end{equation*}
with $\alpha_k(j)$ denoting $\alpha_k(r_j)$. Since the $L^2$ norms on the vertex manifolds of the $u_j$ are bounded in $r_j$, this shows that the coefficients satisfy
  \begin{equation}\label{coeffs.order}
   c_k^\pm(j) \in O(e^{\pm\bullet\alpha_k(j)r_j}) \quad\text{for}\quad j \rightarrow \infty \;.
  \end{equation}
This might seem surprising but we defined the coefficient functions using coordinates with respect to which the boundary of a vertex manifold is given by $t=\pm r_j$. Then, \eqref{coeffs.order} only states that we are in fact dealing with a bounded sequence. Now,
\[ \begin{split}
  \| \big( &(u_j)_- - w_- \big)\big|_{Z_{(v,e)}(s)} \|_{L^2}^2 \\
     &\le \sum_{k > N} \Big(c_k^-(j^\prime) e^{\alpha_k(j^\prime)r_{j^\prime}}
        - \frac{f_{-k}}{2\lambda_k} e^{\lambda_k r_{j^\prime}} \Big)^2 - \Big( c_k^-(j^\prime)
        - \frac{f_{-k}}{2\lambda_k} \Big)^2
     + \epsilon_k^-(j^\prime)c_k^-(j^\prime)^2 e^{2\alpha_k(j^\prime)r_{j^\prime}} \;,
  \end{split} \]
where $\epsilon_k^\pm(j^\prime) = \frac{\lambda_k \pm \alpha_k(j^\prime)}{2\alpha_k(j^\prime)}$. Since $\alpha_k(j^\prime) \rightarrow \lambda_k$ and because of the asymptotics of \thref{props.of.sigma.D.hat}, this sum converges absolutely. What is more, since $c_k^-(j^\prime) \rightarrow \frac{f_{-k}}{2\lambda_k}$, $\epsilon_k^-(j^\prime) \rightarrow 0$ for $j^\prime \rightarrow \infty$ and because $c_k^-(j^\prime)^2 e^{2\alpha_k(j^\prime)r_{j^\prime}}$ is bounded with respect to $j^\prime$, we may conclude convergence of the exponentially decreasing parts.

Since the subsequence $(u_{j^\prime})$ and the matching set $w$ were chosen independently of $s$, this completes the proof.
\qed\end{proof}

%---  surjectivity  ----------------------------------------------------------------------------------------
\section{Surjectivity of Projected Splicing}\label{sec:surjective}

From now on, we assume $r > \lambda_0^{-\frac{1}{1+\varepsilon}}$ since then $|\mu| < r^{-(1 + \varepsilon)}$ implies $|\mu| < \lambda_0$, hence \thref{local.convergence} is applicable to sequences in $E(r) := \spa \big(r,r^{-(1+\varepsilon)}\big)$ respectively $\mathcal{E}_s(r) := \mathcal{R}_s E(r)$. This last theorem completes the proof of \scshape\thref{main.exact}.

\begin{thm}\label{splicing.is.epi} There exists $r_0 > 0$ such that for any $r>r_0$ we have $\dim E(r) \le \dim \mathcal{W}$.
\end{thm}

\begin{proof}We will use a distance argument involving \textbf{Kato's gap} (already used by Cappell, Lee and Miller \cite{clm1} and Nicolaescu \cite{nicolaescu02}). By definition, for two closed, non-trivial subspaces $A$, $B$ of a Banach space $C$, the gap between $A$ and $B$ equals
  \begin{equation*}
    \delta(A,B) = \sup\big\{\, \dist(a,B) \;\big|\; a \in A \;,\; \|a\|_C = 1 \, \big\} \;,
  \end{equation*}
and it is shown that $\delta(A,B)<1$ implies $\dim A \le \dim B$ (see \cite[IV.\S 2]{kato84}, for instance).

For any integer $m>0$ let $\mu_r(1),\dotsc,\mu_r(m)$ denote the first $m$ eigenvalues of $\D(r)$ sorted by increasing absolute value and counted with multiplicity. Then, let $E^m(r)$ be the span of the corresponding eigenforms and $\mathcal{E}_s^m(r) = \mathcal{R}_s E^m(r)$. The proof will be carried out in two steps.

\bigskip\textbf{step 1.} Let $m\in\N$, $s\in\R_+$ and suppose for a sequence $(\tilde{r}_j) \subset \R_+$
  \begin{equation}\label{gap.contra1}
    E^m(\tilde{r}_j) \subset E(\tilde{r}_j) \qquad\text{for all}\quad j \in \N \quad\text{and}\quad
    \lim_{j\rightarrow\infty} \tilde{r}_j = \infty \;.
  \end{equation}
We will show that the gap $\delta\big(\mathcal{E}_s^m(\tilde{r}_j),\mathcal{R}_s\mathcal{W}\big)$ vanishes as $j\rightarrow\infty$. Assume this is wrong, i.e. there is a subsequence $(r_j) \subset (\tilde{r}_j)$ and a constant $c>0$ such that for all $j \in \N$
  \begin{equation}\label{gap.contra2}
    \delta\big(\mathcal{E}_s^m(r_j),\mathcal{R}_s\mathcal{W}\big) = \sup\big\{\,
    \dist( u, \mathcal{R}_s\mathcal{W} ) \;\big|\; u \in \mathcal{E}_s^m(r_j), \|u\|_{L^2(s)}=1 \,\big\} > c \;,
  \end{equation}
where in subscripts we put $L^2(s) = L^2\big(F^0(s)\big)$. Given any sequence $u_j = \mathcal{R}_s \tilde{u}_j$, with $\tilde{u}_j \in E^m(r_j)$ and $\|u_j\|_{L^2(s)}=1$, we may decompose $u_j$ as
  \[ u_j = \sum_{k=1}^m u_j(k) \;, \]
where $u_j(k)=\mathcal{R}_s \tilde{u}_j(k)$ is the restriction of an eigenform $\tilde{u}_j(k)$ of $\D(r)$ with eigenvalue $\mu_{r_j}(k)$. Then, for all $1 \le k \le m$, the sequences $\big(\,\tilde{u}_j(k)\,\big)$ satisfy the assumptions of \thref{local.convergence}. Hence, there are matching sets $w(k)$ so that on suitable subsequences
  \[ \big\| u_j(k) - \mathcal{R}_s w(k) \big\|_{L^2(s)} \longrightarrow 0 \;. \]
Consecutively picking subsequences, we obtain a subsequence $j^\prime$ on which -- for $j^\prime$ sufficiently large:
  \[ \big\| u_{j^\prime} - \mathcal{R}_s w \big\|_{L^2(s)}
     = \big\| \sum_{k=0}^m u_{j^\prime}(k) - \mathcal{R}_s w(k) \big\|_{L^2(s)} \le m \, \delta < c \;, \]
for some $\delta < \frac{c}{m}$. contradicting \eqref{gap.contra2}. Therefore,
  \begin{equation}\label{gap.step1}
   \delta\big( \mathcal{E}^m_s(\tilde{r}_j),\mathcal{R}_s\mathcal{W} \big) \longrightarrow 0 \;,
  \end{equation}
for all sequences $\tilde{r}_j$ satisfying \eqref{gap.contra1}.

\bigskip\textbf{step 2.} Assume the theorem is wrong. Then, there is a sequence $r_j$ with $\lim_j r_j = \infty$ on which $\dim E(r_j) \ge m := \dim\mathcal{W} + 1$. But $\dim E(r_j) \ge m$ implies $E^m(r_j) \subset E(r_j)$, so step 1 and the injectivity of $\mathcal{R}_s$ when acting on $\mathcal{W}$ respectively $E^m(r_j)$ give $m = \dim E^m(r_j) \le \dim\mathcal{W}$, a contradiction. Hence, the theorem is proven.
\qed\end{proof}

%---  references  ------------------------------------------------------------------------------------------

\bibliographystyle{myamsalpha}
\bibliography{hodge_coho}

%-----------------------------------------------------------------------------------------------------------
\end{document}